\documentclass[letter,10pt]{article} \usepackage[utf8]{inputenc} \usepackage{amsfonts}
\usepackage{amsthm} \usepackage{fancyhdr} \usepackage{comment} \usepackage{times}
\usepackage{amsmath} \usepackage{amsrefs} \usepackage{mathtools}
\usepackage{esint} \usepackage{changepage} \usepackage{amssymb} \usepackage{graphicx}
\usepackage{tikz} \usepackage{tikz-cd} \usepackage{titlesec} \usetikzlibrary{arrows}
\usepackage[mathscr]{euscript}  \theoremstyle{definition}
\newtheorem{theorem}{Theorem}[section]

\newtheorem{proposition}[theorem]{Proposition}

 \newcommand{\R}{\mathbb{R}} \newcommand{\C}{\mathbb{C}}
\newcommand{\Z}{\mathbb{Z}} \newcommand{\N}{\mathbb{N}}

\newcommand{\supp}{\operatorname{supp}} \newcommand{\dist}{\operatorname{dist}}
 
 \newcommand{\Iso}{\operatorname{Iso}}

\title{Extremizers for Adjoint Restriction to Pairs of Translated Paraboloids} \author{James Tautges}

\begin{document}
	
\maketitle
	
\begin{abstract}
    Consider the adjoint restriction inequality associated with the hypersurface
    $\{(\tau, \xi) \in \R^{d+1} : \tau = |\xi|^2\} \cup \{(\tau, \xi) \in \R^{d+1} : \tau - \tau_0 =
    |\xi - \xi_0|^2\}$ for any $(\tau_0, \xi_0) \neq 0$. We prove that extremizers do not
    exist for this inequality and fully characterize extremizing sequences in terms of extremizers
    for the adjoint restriction inequality for the paraboloid.
\end{abstract}
	
\section{Introduction}
	
Fix $d \in \N$ and define
\begin{equation}
    \mathcal{E} f(t,x) = \int e^{i (t,x)(|\xi|^2, \xi)} f(\xi) d\xi.
\end{equation}
It is conjectured that
\begin{equation}\label{first sharp const statement}
    \sup_{f \in L^p} \frac{\|\mathcal{E}f\|_q}{\|f\|_p} = A_p < \infty
\end{equation}
for $q > p$ and $q = \frac{d+2}{d}p'$. We will call $p$ and $q$ for which \eqref{first sharp const
  statement} holds ``valid."

There are many recent results about the class of functions $f$ for which
$\|\mathcal{E}f\|_q = A_p\|f\|_p$ and the related question of sequences $\{f_n\}$ such that
$\lim \|\mathcal{E}f_n\|_q/\|f_n\|_p = A_p$. We will call such an $f$ an extremizer for
$\mathcal{E}$ and such $\{f_n\}$ an extremizing sequence. Extremizers were shown to exist in the
case $p=2$ in all dimensions by Shao (\cite{Shao}) and this was extended to all valid $p,q$ in the
interior of the set of all valid $p,q$ by Stovall (\cite{Stovall}). Those results also prove that
extremizing sequences are precompact modulo the symmetries of $\mathcal{E}$, which will be an
important tool in this paper.

As for the value of $A_p$ and the extremizers themselves, Foschi proved that Gaussians
are the unique extremizers for $\mathcal{E}$ for $p = 2$ and $d \in \{1,2\}$, conjecturing that this
was the case for all dimensions (\cite{Foschi 2007}). Christ and Quilodran (\cite{Christ Quilodran})
showed that this conjecture is essentially sharp by proving that Gaussians are only critical points
for the functional $f \mapsto \|\mathcal{E}f\|_q/\|f\|_p$ if $p = 2$. There are similar results
known for restriction to the sphere (e.g.\ \cite{Christ Shao}, \cite{Foschi 2015}, \cite{Frank Lieb
  Sabin}). See Foschi and e Silva's survey (\cite{Foschi e Silva 2017}) for a more comprehensive
collection of known results.

This paper will deal with the related operator
\[
    \mathcal{E}f(t,x) + \mathcal{E}_{(\tau_0, \xi_0)} g(t,x) := \int e^{i(t,x)\cdot(|\xi|^2, \xi)}f(\xi)d\xi + \int e^{i(t,x)\cdot(|\xi - \xi_0|^2 + \tau_0, \xi)} g(\xi)d\xi,
\]
the extension operator associated to a pair of translated paraboloids.
	
\subsection*{Symmetries and Definitions}
We define subgroups $\mathbf{S} \subset \Iso(L^p(\R^d))$ and
$\mathbf{T} \subset \Iso(L^q(\R^{d+1}))$, which are related by
$\mathcal{E} \circ \mathbf{S} = \mathbf{T} \circ \mathcal{E}$, generated by isometries that generate
non-compact subgroups:
\begin{equation*}
    \begin{array}{lll}
      & Sf(\xi) & T\mathcal{E}f(t,x)\\
      \text{Scaling} & \lambda^{d/p}f(\lambda \xi) & \lambda^{-(d+2)/q}\mathcal{E}f(\lambda^{-2}t, \lambda^{-1} x) \\
      \text{Frequency Translation} & f(\xi - \widetilde{\xi}) & e^{i(t|\widetilde{\xi}|^2 + x\cdot \widetilde{\xi})}\mathcal{E}f(t, x+2t\widetilde{\xi}) \\
      \text{Spacetime Translation} & e^{i(t_0, x_0)(|\xi|^2, \xi)}f(\xi) & \mathcal{E}f(t + t_0, x + x_0).
    \end{array}
\end{equation*}
	
We can write any symmetry $S \in \mathbf{S}$ as
\begin{equation}\label{canonical symmetry S+}
    Sf(\xi) = \lambda^{d/p}e^{i(t_0, x_0)(|\lambda\xi - \widetilde{\xi}|^2, \lambda\xi - \widetilde{\xi})}f(\lambda\xi - \widetilde{\xi}),
\end{equation}
and the corresponding $T \in \mathbf{T}$ as
\begin{equation}\label{canonical symmetry T+}
    TF(t,x) = \lambda^{-(d+2)/q}e^{i(\lambda^{-2}t|\widetilde{\xi}|^2 + \lambda^{-1}x\cdot\widetilde{\xi})}F(\lambda^{-2}t + t_0, \lambda^{-1}x + x_0 + 2\lambda^{-2}t\widetilde{\xi}),
\end{equation}
for some $\lambda \in \R^+$, $(t_0, x_0) \in \R \times \R^d$, and $\widetilde{\xi} \in \R^d$.
	
Let
\begin{equation*}
    \mathcal{E}_{(\tau_0, \xi_0)} f(t,x) = \int e^{i(t,x)(|\xi-\xi_0|^2 + \tau_0, \xi)}f(\xi)d\xi
\end{equation*}
be the extension operator associated with the surface
$P_{(\tau_0,\xi_0)} = \{(\tau, \xi) : \tau = |\xi - \xi_0|^2 + \tau_0\}$. Then, although
$\mathcal{E}_{(\tau_0, \xi_0)} \circ \mathbf{S} \neq \mathbf{T} \circ \mathcal{E}_{(\tau_0,
  \xi_0)}$, the generators of $\mathbf{T}$ pass through $\mathcal{E}_{(\cdot,\cdot)}$ as
\begin{equation}\label{translated symmetries}
    \begin{split}
        \lambda^{-(d+2)/q}\mathcal{E}_{(\tau_0, \xi_0)}g(\lambda^{-2}t, \lambda^{-1} x) &= \mathcal{E}_{(\lambda^{-2}\tau_0, \lambda^{-1} \xi_0)}(\lambda^{d/p}g(\lambda\xi))(t,x), \\
        e^{i(t|\widetilde{\xi}|^2 + x\cdot \widetilde{\xi})}\mathcal{E}_{(\tau_0, \xi_0)}g(t, x+2t\widetilde{\xi}) &= \mathcal{E}_{(\tau_0 + 2\xi_0 \cdot \widetilde{\xi}, \xi_0)}(g(\xi - \widetilde{\xi}))(t, x), \\
        \mathcal{E}_{(\tau_0, \xi_0)}g(t + t_0, x + x_0) &= \mathcal{E}_{(\tau_0, \xi_0)}(e^{i(t_0,
          x_0)(|\xi - \xi_0|^2 + \tau_0, \xi)}g(\xi))(t,x).
    \end{split}
\end{equation}
Hence, for every $T \in \mathbf{T}$ written in the form of \eqref{canonical symmetry T+} we have
\begin{multline}\label{canonical symmetry T'}
    T\mathcal{E}_{(\tau_0, \xi_0)} g(t,x) \\
    = \mathcal{E}_{(\lambda^{-2}(\tau_0 + 2\xi_0\cdot \widetilde{\xi}), \lambda^{-1}\xi_0)}
    [\lambda^{d/p}e^{i(t_0, x_0)\cdot(|\lambda\xi - \widetilde{\xi} - \xi_0|^2 + \tau_0, \lambda\xi -
      \widetilde{\xi})}g(\lambda\xi - \widetilde{\xi})](t,x).
\end{multline}

\subsection*{Main Results}
	
From now on we will assume $p,q$ are exponents such that $A_p < \infty$.
	
\begin{theorem}\label{translated theorem}
    Assume $|\tau_0| + |\xi_0| \neq 0$.
    \begin{enumerate}
        \item
        \[
            \sup_{f,g \in L^p} \frac{\|\mathcal{E}f + \mathcal{E}_{(\tau_0, \xi_0)}g\|_q}{(\|f\|_p^p
              + \|g\|_p^p)^{1/p}} = 2^{1/p'}A_p
        \]
        \item For all $f, g \in L^p$,
        \[
            \frac{\|\mathcal{E}f + \mathcal{E}_{(\tau_0,\xi_0)}g\|_q}{(\|f\|_p^p + \|g\|_p^p)^{1/p}}
            < 2^{1/p'}A_p.
        \]
        In other words, $\mathcal{E} + \mathcal{E}_{(\tau_0,\xi_0)}$ has no extremizers.
        \item If the sequence $(f_n,g_n) \subset L^p \times L^p$ extremizes
        $\mathcal{E} + \mathcal{E}_{(\tau_0,\xi_0)}$, then there exists a subsequence in $n$ along
        which
        \begin{equation}\label{sequence repn}
            f_n(\xi) = S_nf(\xi) + r_n(\xi) \quad \text{and} \quad g_n(\xi) = S_nf(\xi) + w_n(\xi)
        \end{equation}
        such that $f$ extremizes $\mathcal{E}$, $\{S_n\} \subset \mathbf{S}$ written in the form of
        \eqref{canonical symmetry S+} with parameters $\lambda, \widetilde{\xi}, t_0, x_0$ that satisfy
        \begin{enumerate}
            \item $\lambda_n \rightarrow \infty$,
            \item $\lambda_n^{-2}\xi_0 \cdot\widetilde{\xi}_n \rightarrow 0$, and
            \item $\lambda_n t_n \xi_0 \rightarrow 0$,
        \end{enumerate}
        and $\|r_n\|_p + \|w_n\|_p \rightarrow 0$.
    \end{enumerate}
\end{theorem}

We prove Theorem \ref{translated theorem} part 1 by considering a sequence of dilates of an
extremizer for $\mathcal{E}$. Next, we use uniform convexity to prove that for any extremizing
sequence $\{(f_n, g_n)\}$, $\|\mathcal{E}f_n - \mathcal{E}_{(\tau_0, \xi_0)}g_n\|_q \rightarrow
0$. Using the separation and transversality of the paraboloids, we also prove that
$f_n \rightharpoonup 0$ and $g_n \rightharpoonup 0$ and deduce Theorem \ref{translated theorem} part
2.  Finally, we use the results in \cite{Stovall} to find symmetries $\{S_n\}$ such that
$S_nf_n \rightarrow f$ in $L^p$. Theorem \ref{translated theorem} part 3 follows from a direct
computation.
	
\subsection*{Acknowledgements}
	
This project was suggested and overseen by Betsy Stovall and supported in part by NSF
DMS-1653264. The author would like to thank her for many helpful conversations and invaluable
guidance in the writing of this paper.
	
\section{Operator Norm}

First, we prove boundedness.
\begin{proposition}\label{general estimate}
    For $f,g \in L^p(\R^d)$,
    \begin{equation}\label{upper bound}
        \left\|\mathcal{E}f + \mathcal{E}_{(\tau_0, \xi_0)}g\right\|_q \leq 2^{1/p'}A_p\left(\left\|f\right\|_p^p + \left\|g\right\|_p^p\right)^{1/p}.
    \end{equation}
\end{proposition}
	
\begin{proof}
    Since
    $\|\mathcal{E}_{(\tau_0, \xi_0)}g(t,x)\|_q = \|e^{i(t,x)\cdot(\tau_0, \xi_0)}\mathcal{E}[g(\cdot
    + \xi_0)](t,x)\|_q \leq A_p\|g\|_p$,
    \begin{equation}\label{upper bound comp}
        \begin{split}
            \left\|\mathcal{E}f + \mathcal{E}_{(\tau_0, \xi_0)}g\right\|_q &= \left\|\mathcal{E}f + \mathcal{E}_{(\tau_0, \xi_0)}g\right\|_q \leq \left\|\mathcal{E}f\right\|_q + \left\| \mathcal{E}_{(\tau_0, \xi_0)}g \right\|_q \\
            &\leq A_p\left( \|f\|_p + \|g\|_p\right) \leq 2^{1/p'}A_p \left( \|f\|_p^p +
                \|g\|_p^p\right)^{1/p}.
        \end{split}
    \end{equation}
\end{proof}

Next, we construct an extremizing sequence for
$\mathcal{E} + \mathcal{E}_{(\tau_0, \xi_0)}: L^p \times L^p \rightarrow L^q$ to show that
$2^{1/p'}A_p$ is the sharp constant.

\begin{proposition}\label{existence}
    Let $f \in L^p$ be such that $\|\mathcal{E}f\|_q = A_p\|f\|_p$ and let
    $f_\lambda(\xi) := \lambda^{d/p}f(\lambda\xi)$. Then
    \[
        \lim_{\lambda \rightarrow 0} \frac{\|\mathcal{E}f_\lambda +
          \mathcal{E}_{(\tau_0,\xi_0)}f_\lambda\|_q}{2^{1/p}\|f\|_p} = 2^{1/p'}A_p.
    \]
\end{proposition}

\begin{proof}
    Let $f \in L^p$ be such that $\|f\|_p = 1$ and $\|\mathcal{E}f\|_q = A_p$. By the scaling
    properties of $\mathcal{E}$ we know that $\|f_\lambda\|_p = \|f\|_p$ and
    $\|\mathcal{E}f_\lambda\|_q = \|\mathcal{E}f\|_q$ for all $\lambda > 0$. The identity
    \begin{equation*}
        \mathcal{E}_{(\tau_0, \xi_0)} f_\lambda (t, x) = \lambda^{-(d+2)/q}\mathcal{E}_{(\lambda^2\tau_0, \lambda\xi)}f(\lambda^{-2}t, \lambda^{-1}x).
    \end{equation*}
    implies
    $\left\|\mathcal{E}_{(\tau_0, \xi_0)}f_\lambda + \mathcal{E}f_\lambda\right\|_q =
    \left\|\mathcal{E}_{(\lambda^2\tau_0, \lambda\xi_0)}f + \mathcal{E}f\right\|_q$. Moreover,
    \begin{equation*}
        \mathcal{E}_{(\lambda^2\tau_0, \lambda\xi_0)}f(t,x) = e^{it\lambda^2(|\xi_0|^2 + \tau_0)} \mathcal{E}f(t, x - 2\lambda t\xi_0).
    \end{equation*}
    
    By approximating $\mathcal{E}f$ in $C^\infty_{cpct}$ and applying dominated convergence, it is
    clear that $\|\mathcal{E}_{(\lambda^2\tau_0, \lambda\xi_0)}f - \mathcal{E}f\|_q \rightarrow
    0$. Since $f$ is an extremizer, this implies that
    $\|\mathcal{E}_{(\tau_0, \xi_0)}f + \mathcal{E}f\|_q \rightarrow 2A_p$.
\end{proof}

\begin{proof}[Proof of Theorem \ref{translated theorem} part 1]
    This follows from Proposition \ref{existence} and Proposition \ref{general estimate}.
\end{proof}

\section{Non-Existence of Extremizers}

Now we use the operator norm to understand extremizing sequences. We begin by proving that
extremizing sequences of functions must converge weakly to zero.

\begin{proposition}\label{weak limit}
    Assume that $|\tau_0| + |\xi_0| > 0$. Then for every bounded extremizing sequence $\{(f_n, g_n)\}$
    for $\mathcal{E} + \mathcal{E}_{(\tau_0, \xi_0)}$,
    \begin{enumerate}
        \item $\{f_n\}$ extremizes $\mathcal{E}$ and $\{g_n\}$ extremizes
        $\mathcal{E}_{(\tau_0, \xi_0)}$;
        \item $\|f_n\|_p - \|g_n\|_p \rightarrow 0$ as $n \rightarrow \infty$;
        \item $\|\mathcal{E}f_n - \mathcal{E}_{(\tau_0, \xi_0)}g_n\|_q \rightarrow 0$ as
        $n \rightarrow \infty$; and
        \item $f_n, g_n \rightharpoonup 0$ weakly in $L^p$.
    \end{enumerate}
    
\end{proposition}

\begin{proof}
    Let $\{(f_n, g_n)\} \subset \ell^p(L^p \times L^p)$ such that
    $(\|f_n\|_p^p + \|g_n\|_p^p)^{1/p} = 1$ for all $n$ and
    \begin{equation*}
        \lim_{n\rightarrow \infty} \|\mathcal{E}f_n + \mathcal{E}_{(\tau_0,\xi_0)}g_n\|_q = 2^{1/p'}A_p.
    \end{equation*}
    By \eqref{upper bound comp}, we must have
    \begin{equation}\label{first}
        \lim_{n \rightarrow \infty}\frac{\left\|\mathcal{E}f_n + \mathcal{E}_{(\tau_0, \xi_0)}g_n\right\|_q}{\left\|\mathcal{E}f_n\right\|_q + \left\| \mathcal{E}_{(\tau_0, \xi_0)}g_n \right\|_q} = 1;
    \end{equation}
    \begin{equation}\label{second}
        \lim_{n \rightarrow \infty}\frac{\left\|\mathcal{E}f_n\right\|_q + \left\| \mathcal{E}_{(\tau_0, \xi_0)}g_n \right\|_q}{A_p\left( \|f_n\|_p + \|g_n\|_p\right)} = 1;
    \end{equation}
    and
    \begin{equation}\label{third}
        \lim_{n \rightarrow \infty}\frac{A_p\left( \|f_n\|_p + \|g_n\|_p\right)}{2^{1/p'}A_p \left( \|f_n\|_p^p + \|g_n\|_p^p\right)^{1/p}} = 1.
    \end{equation}
    
    By \eqref{third} and the sharp H\"older inequality, $\|f_n\|_p, \|g_n\|_p \rightarrow 2^{-1/p}$
    proving condition 2. Combining this with \eqref{second} implies condition 1 and, in particular,
    $\|\mathcal{E}f_n\|_q, \|\mathcal{E}_{(\tau_0,\xi_0)}g_n\|_q \rightarrow 2^{-1/p}A_p$. In light
    of \eqref{first}, we see that
    $\|\mathcal{E}f_n + \mathcal{E}_{(\tau_0,\xi_0)}g_n\|_q - \|\mathcal{E}f_n\|_q -
    \|\mathcal{E}_{(\tau_0,\xi_0)}g_n\|_q \rightarrow 0$. Since $L^q$ is uniformly convex
    (\cite[Theorem 2.5]{LiebLoss}), this proves condition 3.
    
    Turning to condition 4, let $\psi = \overline{\widehat{\phi}} \in \mathcal{S}$ and let
    $d\sigma(\tau, \xi)$ be the measure on the paraboloid centered at the origin given by the
    pullback of Lebesgue measure via the projection map $\R \times \R^d \rightarrow \R^d$. By
    definition, $\mathcal{E}f = \widehat{fd\sigma}$ and since $L^q \subset \mathcal{S}'$, we may
    compute
    \begin{equation*}
        \langle \mathcal{E}f_n, \phi\rangle = \langle f_nd\sigma, \widehat{\phi}\rangle = \int f_n(\xi) \psi(|\xi|^2, \xi)d\xi.
    \end{equation*}
    In the same way, let $d\sigma'(\tau, \xi)$ be the measure on the translated paraboloid so that
    \begin{equation*}
        \langle \mathcal{E}_{(\tau_0, \xi_0)} g_n, \phi\rangle = \langle g_nd\sigma', \widehat{\phi}\rangle = \int g_n(\xi)\psi(|\xi - \xi_0|^2 + \tau_0, \xi)d\xi.
    \end{equation*}
    
    Let $\eta \in \R^d$. As long as $(|\eta|^2, \eta) \in P$ isn't on the intersection of the two
    paraboloids, there exists $r > 0$ sufficiently small that $B((|\eta|^2, \eta),r)$ is disjoint
    from $P_{(\tau_0,\xi_0)}$. Let $\psi: \R^{d+1} \rightarrow \R$ be any smooth function supported
    on $B((|\eta|^2, \eta), r)$. Since $\psi(|\xi - \xi_0|^2 + \tau_0, \xi) = 0$ for all
    $\xi \in \R^d$, $\langle\mathcal{E}_{(\tau_0,\xi_0)}g_n, \phi\rangle = 0$ for all $n$ so we can
    apply condition 3 to see that
    \begin{equation*}
        \int f_n(\xi) \psi(|\xi|^2, \xi)d\xi \rightarrow 0.
    \end{equation*}
    Since $\psi(|\cdot|^2, \cdot)$ ranges over all smooth functions supported on $B(\eta, r)$,
    $f_n \rightharpoonup 0$ weakly on a neighborhood of almost every point. Hence
    $f_n \rightharpoonup 0$ weakly on all of $\R^d$ and the statement for $g_n$ follows
    similarly. This proves condition 4.
\end{proof}

\begin{proof}[Proof of Theorem \ref{translated theorem} part 2]
    An extremizing pair $(f, g)$ is a constant extremizing sequence. Therefore non-zero extremizers
    do not exist.
\end{proof}

\section{Characterization of Extremizing Sequences}

Let $\{(f_n, g_n)\} \subset \ell^p(L^p \times L^p)$ be an extremizing sequence for
$\mathcal{E} + \mathcal{E}_{(\tau_0, \xi_0)}$ such that $\|f_n\|_p^p + \|g_n\|_p^p = 1$ for all
$n$. By Proposition \ref{weak limit}, $\{f_n\}$ is an extremizing sequence for $\mathcal{E}$ so by
\cite[Theorem 1.1]{Stovall} there exist $\{S_n\} \subset \mathbf{S}$ such that
$S_n f_n \rightarrow f$ in $L^p$ along some subsequence where $f$ is an extremizer for
$\mathcal{E}$. Let $\{T_n\} \subset \mathbf{T}_+$ be such that
$T_n \circ \mathcal{E} = \mathcal{E} \circ S_n$ for all $n$. Then
\begin{multline*}
    2^{1/p'}A_p = \lim_{n\rightarrow\infty} \|\mathcal{E}f_n + \mathcal{E}_{(\tau_0, \xi_0)}g_n\|_q \\
    = \lim_{n\rightarrow\infty} \|T_n\mathcal{E}f_n + T_n\mathcal{E}_{(\tau_0, \xi_0)}g_n\|_q =
    \lim_{n\rightarrow\infty} \|\mathcal{E}f + T_n\mathcal{E}_{(\tau_0, \xi_0)}g_n\|_q.
\end{multline*}

Since $\|\mathcal{E}f\|_q = A_p\|f\|_p$ and $T_n$ are $L^q$ symmetries, uniform convexity implies
that $T_n\mathcal{E}_{(\tau_0, \xi_0)}g_n \rightarrow \mathcal{E}f$ in $L^q$. From here we can
deduce the behavior of $S_n g_n$.

\begin{proposition}\label{translated fourier support condition}
    Let $f \in L^p(\R^d)$ such that $\|f\|_p = 1$ and $\{(\tau_n, \xi_n)\} \subset \R^{d+1}$. If
    $\lim \tau_n = \tau_0$ and $\lim \xi_n = \xi_0$, then
    \[
        \lim \|\mathcal{E}_{(\tau_0,\xi_0)}f - \mathcal{E}_{(\tau_n,\xi_n)}f\|_q = 0.
    \]
    Conversely, if $f \in L^p(\R^d)$ and $\{g_n\} \subset L^p(\R^d)$ are such that
    $\|f\|_p = \|g_n\|_p = 1$ for all $n$ and
    \[
        \lim \|\mathcal{E}_{(\tau_0,\xi_0)}f - \mathcal{E}_{(\tau_n,\xi_n)}g_n\|_q = 0,
    \]
    then $\lim \tau_n = \tau_0$ and $\lim \xi_n = \xi_0$.
\end{proposition}
\begin{proof}
    Assume that $\lim \tau_n = \tau_0$ and $\lim \xi_n = \xi_0$. First, we rearrange the expression:
    \begin{align*}
      &\|\mathcal{E}_{(\tau_0,\xi_0)}f - \mathcal{E}_{(\tau_n, \xi_n)}f\|_q \\
      &= \left\|\int \left[e^{i(t,x)\cdot(|\xi - \xi_0|^2 + \tau_0,\xi)} - e^{i(t,x)\cdot(|\xi - \xi_n|^2 + \tau_n, \xi)}\right]f(\xi)d\xi \right\|_q \\
      &= \left\|\int \left[e^{it(-2\xi\cdot\xi_0 + |\xi_0|^2 + \tau_0)} - e^{it(-2\xi\cdot\xi_n + |\xi_n|^2 + \tau_n)}\right]e^{i(t,x)\cdot(|\xi|^2,\xi)}f(\xi)d\xi\right\|_q \\
      &= \Bigg\|\int e^{it(-2\xi\cdot\xi_0 + |\xi_0|^2 + \tau_0)} \\
      &\qquad\left[1 - e^{it(-2\xi\cdot(\xi_n-\xi_0) + |\xi_n|^2 - |\xi_0|^2 + \tau_n - \tau_0)}\right]e^{i(t,x)\cdot(|\xi|^2,\xi)}f(\xi)d\xi\Bigg\|_q
    \end{align*}
    Let $\varepsilon > 0$. Take $g \in C^\infty_{cpct}(\R^d)$ such that
    $\|f - g\|_p = O(\varepsilon)$. Let $R>0$ be large enough that $\supp g \subset B(0, R)$ and
    \[
        \Big|\|\mathcal{E}_{(\tau_0,\xi_0)}g - \mathcal{E}_{(\tau_n, \xi_n)}g\|_{L^q(\R^{d+1})} -
        \|\mathcal{E}_{(\tau_0,\xi_0)}g - \mathcal{E}_{(\tau_n, \xi_n)}g\|_{L^q(B(0,R))} \Big| =
        O(\varepsilon).
    \]
    Then by Minkowski's integral inequality and H\"older,
    \begin{align*}
      &\|\mathcal{E}_{(\tau_0,\xi_0)}f - \mathcal{E}_{(\tau_n, \xi_n)}f\|_{L^q(\R^{d+1})} \\
      &= \|\mathcal{E}_{(\tau_0,\xi_0)}g - \mathcal{E}_{(\tau_n, \xi_n)}g\|_{L^q(\R^{d+1})} + O(\varepsilon) \\
      &= \|\mathcal{E}_{(\tau_0,\xi_0)}g - \mathcal{E}_{(\tau_n, \xi_n)}g\|_{L^q(B(0,R))} + O(\varepsilon) \\
      &= \Bigg\|\int e^{it(-2\xi\cdot\xi_0 + |\xi_0|^2 + \tau_0)} \\
      &\left[1 - e^{it(-2\xi\cdot(\xi_n-\xi_0) + |\xi_n|^2 - |\xi_0|^2 + \tau_n - \tau_0)}\right]e^{i(t,x)\cdot(|\xi|^2,\xi)}g(\xi)d\xi\Bigg\|_{L^q(B(0,R))} + O(\varepsilon) \\
      &\leq \|g\|_1 \sup_{|\xi|<R}\|1 - e^{it(-2\xi\cdot(\xi_n-\xi_0) + |\xi_n|^2 - |\xi_0|^2 + \tau_n - \tau_0)}\|_{L^q(B(0,R))} + O(\varepsilon).
    \end{align*}
    Furthermore, since $\frac{d}{d\theta}(1-e^{i\theta}) = -ie^{i\theta}$, the mean value theorem implies that
    \begin{align*}
      &\lim_{n\rightarrow \infty} \sup_{|\xi|<R} \|1 - e^{it(-2\xi\cdot(\xi_n-\xi_0) + |\xi_n|^2 - |\xi_0|^2 + \tau_n - \tau_0)}\|_{L^q(|(t,x)|<R)} \\
      &\lesssim \lim R^{\frac{d+1}{q}}\sup_{|\xi|< R,\; |(t,x)|<R} |1 - e^{it(-2\xi\cdot(\xi_n-\xi_0) + |\xi_n|^2 - |\xi_0|^2 + \tau_n - \tau_0)}| \\
      &\leq \lim R^{\frac{d+1}{q}} \sup |t(-2\xi\cdot(\xi_n-\xi_0) + |\xi_n|^2 - |\xi_0|^2 + \tau_n - \tau_0)| \\
      &\leq \lim R^{\frac{d+1}{q}} R (2R|\xi_n-\xi_0| + ||\xi_n|^2 - |\xi_0|^2| + |\tau_n - \tau_0|) \\
      &= 0.
    \end{align*}
    Hence
    $\limsup_{n\rightarrow\infty} \|\mathcal{E}_{(\tau_0,\xi_0)}f - \mathcal{E}_{(\tau_n,
      \xi_n)}f\|_{L^q(\R^{d+1})} = O(\varepsilon)$ and, taking $\varepsilon \rightarrow 0$,
    \[
        \lim_{n\rightarrow\infty} \|\mathcal{E}_{(\tau_0,\xi_0)}f - \mathcal{E}_{(\tau_n,
          \xi_n)}f\|_{L^q(\R^{d+1})} = 0.
    \]
    
    Conversely, assume $\lim \tau_n \neq \tau_0$ or $\lim \xi_n \neq \xi_0$.
    
    Consider the distance between the two paraboloids over the point $\xi \in \R^d$,
    $h_n(\xi) := |\xi-\xi_0|^2 - |\xi-\xi_n|^2 + \tau_0 - \tau_n$. We rearrange to find
    $h_n(\xi) = 2\xi\cdot(\xi_n - \xi_0) + |\xi_0|^2 + \tau_0 - |\xi_n|^2 - \tau_n$. Let
    $A_n = \{\xi : h_n(\xi) = 0\}$. We claim that for any fixed $d$-ball $R > 0$,
    \begin{align}\label{separation ineq}
      \limsup_n \inf_{\{\xi : |\xi| < R, \;\dist(\xi, A_n) > s\}} |h_n(\xi)| =: c_{s,R} > 0.
    \end{align}
    Indeed, $c_{s,R} \geq \limsup_n 2s|\xi_n - \xi_0|$ by differentiating $h_n$. If this quantity is
    zero, then $\lim \xi_n = \xi_0$ and therefore $\lim \tau_n \neq \tau_0$ in this case. Let
    $\varepsilon > 0$. There exists an $N > 0$ such that
    \[
        \sup_{\{\xi : |\xi|< R\}} \left| |\xi - \xi_0|^2 - |\xi - \xi_n|^2\right| < \varepsilon
    \]
    for all $n > N$, so
    \[
        c_{s,R} \geq \limsup_n \inf_{\{\xi : |\xi| < R\}} |h_n(\xi)| \geq \limsup_n |\tau_n - \tau_0|
        - \varepsilon.
    \]
    Since $\limsup_n |\tau_n - \tau_0| > 0$ in the case we're considering, we can take $\varepsilon$
    small enough to show that $c_{s,R} > 0$, proving \eqref{separation ineq}.
    
    We would like to construct a function $\Psi \in C^\infty_{cpct}(\R^{d+1})$ such that after
    passing to a subsequence in $n$,
    \begin{enumerate}
        \item $|\langle fd\sigma, \Psi\rangle| > \frac{1}{2}$, and
        \item $\lim_n \sup_{\xi \in B(0,R)} |\Psi(\xi, |\xi - \xi_n|^2 + \tau_n)| = 0$.
    \end{enumerate}
    Let $\eta \in \C^\infty_{cpct}(\R)$ be a non-negative bump function with
    $\supp \eta \subset B(0,1)$ and $\eta|_{B(0,1/2)} \equiv 1$. Also let
    $\Phi \in C^\infty_{cpct}(\R^d)$ be such that $\|\Phi\|_{p'} = 1$ and
    $|\langle f, \Phi\rangle| > \frac{3}{4}$, and take $R > 0$ so that $\supp \Phi \subset
    B(0,R)$. Note that, since the zero set of $h$ is a $(d-1)$-hyperplane,
    \[
        |\{\xi : \dist(\xi, A_n) < s\} \cap B(0,R)| \leq c_{d-1}R^{d-1}s
    \]
    where $c_{d-1}$ is a dimensional constant. Therefore, there exists an $s_0 > 0$ such that
    \[
        \|\Phi - \left[1-\eta\left(\frac{\dist(\xi, A_n)}{2s_0}\right)\right]\Phi\|_{p'} <
        \frac{1}{4},
    \]
    and hence
    \begin{equation*}
        |\langle f, \left[1-\eta\left(\frac{\dist(\xi, A_n)}{2s_0}\right)\right]\Phi\rangle| > \frac{1}{2}
    \end{equation*}
    for all $n$. Now let
    \[
        \Psi_n(\tau, \xi) := \eta\left(3\frac{\tau - \tau_0 - |\xi-\xi_0|^2}{c_{s_0,R}}\right)
        \left[1-\eta\left(\frac{\dist(\xi, A_n)}{2s_0}\right)\right]\Phi(\xi).
    \]
    By \eqref{separation ineq}, $\Psi_n(|\xi-\xi_n|^2 + \tau_n, \xi) = 0$ for all $\xi \in
    B(0,R)$. In addition, $|\langle fd\sigma, \Psi_n\rangle| > \frac{1}{2}$ by construction.
    
    Since the space of all $(d-1)$-planes intersecting the $d$-ball $\overline{B(0,R+10s_0)}$ is
    compact, either $A_n \cap \overline{B(0, R+10s_0)} = \emptyset$ for sufficiently large $n$, or
    there exists a hyperplane $A$ such that for a subsequence, $A_n \rightarrow A$ in the sense that
    \[
        \lim_n \sup_{\substack{\xi \in A\cap \overline{B(0,R+10s_0)}, \\ \zeta \in A_n\cap
            \overline{B(0,R+10s_0)}}} |\xi - \zeta| = 0.
    \]
    In the first case, let
    \[
        \Psi(\tau, \xi) := \eta\left(3\frac{\tau - \tau_0 - |\xi-\xi_0|^2}{c_{s_0,R}}\right)
        \Phi(\xi),
    \]
    and in the second, pass to the subsequence mentioned above and let
    \[
        \Psi(\tau, \xi) := \eta\left(3\frac{\tau - \tau_0 - |\xi-\xi_0|^2}{c_{s_0,R}}\right)
        \left[1-\eta\left(\frac{\dist(\xi, A)}{2s_0}\right)\right]\Phi(\xi).
    \]
    In either case, we see that $\Psi_n \rightarrow \Psi$ in $C^\infty_{cpct}$. Condition 1 holds
    since $fd\sigma \in \mathcal{S}'(\R^{d+1})$ and hence
    $\langle fd\sigma, \Psi_n\rangle \rightarrow \langle fd\sigma, \Psi\rangle$. Condition 2 holds
    by the triangle inequality and the fact that it's satisfied for each $\Psi_n$ individually.
    
    By Plancherel and condition 1,
    $|\langle \mathcal{E}_{(\tau_0,\xi_0)}f, \widetilde{\widehat{\Psi}}\rangle| > \frac{1}{2}$. On
    the other hand, by condition 2,
    $\lim_n |\langle \mathcal{E}_{(\tau_n,\xi_n)}g_n, \widetilde{\widehat{\Psi}}\rangle| = 0$. Thus
    \begin{equation*}
        \lim_{n\rightarrow \infty} \|\mathcal{E}_{(\tau_0,\xi_0)} f - \mathcal{E}_{(\tau_n,\xi_n)} g_n\|_q \geq \frac{1}{2\|\widehat{\Psi}\|_{q'}} \neq 0,
    \end{equation*}
    which is a contradiction and proves the proposition.
\end{proof}

\begin{proof}[Proof of Theorem \ref{translated theorem} part 3]
    By Proposition \ref{weak limit},
    \begin{equation*}
        \|\mathcal{E}f_n\|_q, \|\mathcal{E}_{(\tau_0, \xi_0)}g_n\|_q \rightarrow A_p\quad \text{and}\quad \|\mathcal{E}f_n + \mathcal{E}_{(\tau_0, \xi_0)}g_n\|_q \rightarrow 2A_p,
    \end{equation*}
    so uniform convexity implies that
    $\|\mathcal{E}f_n - \mathcal{E}_{(\tau_0, \xi_0)}g_n\|_q \rightarrow 0$.
    
    Let $\phi \in C^\infty_{cpct}(\R^d)$. Set
    \[
        L = \sup_{\xi \in \supp \phi} \left| |\xi-\xi_0|^2 + \tau_0 - |\xi|^2\right|
    \]
    and let $\eta \in C^\infty_{cpct}(\R)$ be such that $\eta|_{B(0, 2L)} \equiv 1$. Now let
    $\psi(\tau, \xi) = \phi(\xi)\eta(\tau - |\xi|^2)$.
    
    By construction of $\eta$,
    \begin{multline*}
        \lim_{n\rightarrow \infty} |\langle f_n - g_n, \phi\rangle_{L^2(\R^d)}| = \lim_n |\langle f_nd\sigma - g_nd\sigma', \psi\rangle_{L^2(\R^{d+1})}| \\
        = \lim_n |\langle \mathcal{E}f_n - \mathcal{E}_{(\tau_0,\xi_0)}g_n, \widehat{\psi}\rangle|
        \leq \lim_n \|\mathcal{E}f_n - \mathcal{E}_{(\tau_0, \xi_0)}g_n\|_q \|\widehat{\psi}\|_{q'}
        = 0,
    \end{multline*}
    which proves that $f_n - g_n \rightharpoonup 0$. Indeed, since
    $\|f_n\|_p - \|g_n\|_p \rightarrow 0$ by Proposition \ref{weak limit}, uniform convexity implies
    that the convergence is strong, $\|f_n - g_n\|_p \rightarrow 0$. Since $f_n$ is an extremizing
    sequence, by \cite[Theorem 1.1]{Stovall} there exists a subsequence in $n$,
    $\{S_n\} \subset \mathbf{S}$, and an extremizer $f \in L^p$ such that $S_n f_n \rightarrow f$ so
    we have $\|f - S_ng_n\|_p \rightarrow 0$ along this subsequence as well, i.e.\ \eqref{sequence repn} holds.
    
    Recall that the symmetries $T_n \mathcal{E}_{(\tau_0, \xi_0)} g_n$ can be expressed in the form
    of \eqref{canonical symmetry T'} as
    \begin{multline}\label{translated symmetry}
        \mathcal{E}_{(\lambda^{-2}_n(\tau_0 + 2\xi_0\cdot \widetilde{\xi}_n), \lambda^{-1}_n\xi_0)} \lambda_n^{d/p}e^{i(t_n, x_n)\cdot(|\lambda_n\xi - \widetilde{\xi}_n - \xi_0|^2 + \tau_0, \lambda_n\xi - \widetilde{\xi}_n)}g_n(\lambda_n\xi - \widetilde{\xi}_n)(t,x) \\
        = \mathcal{E}_{(\lambda^{-2}_n(\tau_0 + 2\xi_0\cdot \widetilde{\xi}_n),
          \lambda^{-1}_n\xi_0)}(e^{it_n(-2\lambda_n\xi\cdot\xi_0 + 2\xi_0\cdot\widetilde{\xi}_n + |\xi_0|^2 +
          \tau_0)} S_ng_n)(t, x)
    \end{multline}
    for $\lambda_n \in \R^+$, $(t_n, x_n) \in \R \times \R^d$ and $\widetilde{\xi}_n \in \R^d$. Proposition
    \ref{translated fourier support condition} immediately implies that
    $(\lambda_n^{-2}(\tau_0 + 2\xi_0\cdot\widetilde{\xi}_n),\lambda_n^{-1}\xi_0) \rightarrow 0$, which also
    implies $\lambda_n \rightarrow \infty$ and $\lambda_n^{-2}\xi_0\cdot\widetilde{\xi}_n \rightarrow 0$ since
    we assumed $|\tau_0| + |\xi_0| \neq 0$. Since $S_ng_n \rightarrow f$ in $L^p$, by the triangle
    inequality and the forward part of Proposition \ref{translated fourier support condition},
    \[
        \left\|\mathcal{E}f - \mathcal{E}(e^{it_n(-2\lambda_n\xi\cdot \xi_0 + 2\xi_0\cdot\widetilde{\xi}_n +
              |\xi_0|^2 + \tau_0)} f)(t, x)\right\|_q \rightarrow 0
    \]
    as well. Since $\R/(2\pi\Z)$ is compact, we may pass to a subsequence along which
    $\theta := \lim t_n(2\xi_0\cdot\widetilde{\xi}_n + |\xi_0|^2 + \tau_0)/(2\pi\Z)$ exists. By rearranging,
    applying the boundedness of $\mathcal{E}$, and invoking dominated convergence,
    \[
        \left\| \mathcal{E}e^{i\theta}f(t,x-2\lambda_nt_n\xi_0) -
            \mathcal{E}(e^{it_n(-2\lambda_n\xi\cdot \xi_0 + 2\xi_0\cdot\widetilde{\xi}_n + |\xi_0|^2 + \tau_0)}
            f)\right\|_q \rightarrow 0.
    \]
    The two preceding limits imply that
    \[
        \left\| \mathcal{E}e^{i\theta}f(t, x-2\lambda_n t_n \xi_0) - \mathcal{E}f\right\|_q
        \rightarrow 0.
    \]
    Hence the sequence $\{\lambda_n t_n\}$ must be bounded and there exists a subsequence along
    which $c := \lim \lambda_n t_n$ exists. Since $\lambda_n \rightarrow \infty$, $\lim t_n = 0$. As
    translation is continuous in $L^q$,
    \[
        \left\|\mathcal{E}f - \mathcal{E}e^{i\theta}f(t,x-2c\xi_0) \right\|_q = 0.
    \]
    From this is is clear that $c = 0$ as $|\mathcal{E}f(t,x)| = |\mathcal{E}f(t,x - 2c\xi_0)|$
    would imply that $\mathcal{E}f \equiv 0$ since $\mathcal{E}f \in L^q$. We now also see that
    $e^{i\theta}=1$ by strict convexity.
    
    This proves the claim.
\end{proof}

\end{document}